\documentclass[10pt]{amsproc}

\usepackage{amsmath}
\usepackage{graphicx}
\usepackage{amssymb}
\usepackage{amsthm}
\usepackage{verbatim}
\usepackage{microtype}

\newtheorem{theorem}{Theorem}
\newtheorem{lemma}[theorem]{Lemma}

\newtheorem{proposition}[theorem]{Proposition}

\theoremstyle{definition}

\newtheorem{definition}[theorem]{Definition}

\begin{document}

\title{ Plane Partition Polynomial Asymptotics  }
\author{Robert P. Boyer}
\address{Department of Mathematics, Drexel University USA}
\email{rboyer@drexel.edu}
\author{Daniel T. Parry}
\address{Department of Mathematics, Drexel University USA}
\email{dtp23@drexel.edu}


\begin{abstract}
The plane partition polynomial $Q_n(x)$ 
is the polynomial of degree $n$ whose coefficients count the number of plane partitions of $n$
indexed by their trace.
Extending classical work of E.M. Wright, we develop the asymptotics of these polynomials inside the unit disk using the circle method. 
\end{abstract}
\maketitle

\section{Introduction}

A plane partition $\pi$ of a positive integer $n$ is an
array $[ \pi_{i,j}]$ of nonnegative integers such that $\sum \pi_{i,j} =n$ while its trace is the sum of the diagonal entries $\sum \pi_{i,i}$.
The asymptotics of the plane partition numbers $PL(n)$, the number of all plane partitions of $n$, was found by Wright \cite{WrightPlanePartitions}  in 1931
using the circle  and saddle point methods.
In this paper, we study the asymptotics of polynomial versions  $Q_n(x)$ of the plane partitions of $n$ given by

\begin{definition}\label{def:Q_n}
Let
$Q_n(x)$
be the $n$-th degree polynomial given by $ \sum_{k=1}^n pp_k(n) x^k$ where $pp_k(n)$ is the number of plane partitions of $n$ with trace $k$.
These  polynomials  have  generating function
\begin{equation}\label{eq:generating_function}
P(x,u) = \prod_{m=1}^\infty \frac{ 1}{ (1- xu^m)^m} = 1 + \sum_{n=1}^\infty Q_n(x) u^n 
\end{equation}
(see \cite{Andrews_book,Stanley}). 
\end{definition}

In adapting  the circle method to develop the asymptotics of these polynomials  (see Section \ref{section:asymptotics}), 
we needed  to introduce the sequence  $\{L_k(x)\}$  of  functions 
to describe the dominant contributions to their asymptotics  where
\begin{equation}\label{def:L_k}
L_k(x) = \frac{1}{k} \, \sqrt[3]{ 2 Li_3(x^k) }
\end{equation}
and $Li_3(x)$ is the trilogarithm function given by $\sum_{n=1}^\infty x^n / n^3$.

An important  first step in the asymptotic analysis of the polynomials was obtained in our paper
\cite{Boyer_Parry_phase} where we determined exactly when $\Re L_m(x)$ ($m=1,2)$
dominates $\Re L_k(x)$, for $k \neq m$, inside the unit disk (see  \cite{Boyer_Parry_phase}).

A primary motivation for us to develop the asymptotics was to find the limiting behavior of the zeros
of the plane partition polynomials which is described in detail in our paper  \cite{Boyer_Parry_EJC}.

\section{Factorization of the Generating Function $P(x,u)$}

We are interested in the behavior of $P(x,u)$ in a neighborhood of $u= e^{ 2 \pi i h/ k}$ inside the unit disk $\mathbb D$, where
$h$ and $k$ are relatively prime. To start
 we expand the logarithm of the generating function $\ln P(w,u) $
\[
\ln P(w,u) = \sum_{\ell=1}^\infty \frac{ x^\ell}{  \ell} \frac{ u^\ell }{ (1-u^\ell)^2} .
\]
for $x,u \in {\mathbb D}$.
Next with  $u = e^{ - w + 2 \pi i h/k}$, with  $\Re(w) >0$,
we  introduce two functions 
\begin{align}\label{eq:defAB}
A_{h,k}(x,w)   = 
\sum_{  k\nmid \ell}
 \frac{x^\ell}{ \ell} \frac{ e^{-\ell w+ 2 \pi i \ell h/k}     }{  (1- e^{-\ell w+ 2 \pi i \ell h/k  })^2  }        , \quad
B_{h,k}(x,w)=
\sum_{\ell=1}^\infty \frac{ x^{k \ell}  }{ k \ell} \frac{ e^{ - \ell k w}  }{ ( 1- e^{ - \ell k w} )^2  }.
\end{align}
Then the generating function decomposes as
\[
\ln[  P( x, e^{ - w + 2 \pi i h/k} ) ]
=
A_{h,k}(x, x , w) + B_{ h, k}(x,w) .
\]
Note that when  $k=1$, $A_{h,k}(x,w)=0$ and that
\[
A_{h,k}(x,  0) = 
- \frac{1}{ 4}  \sum_{ k \nmid \ell} \frac{ x^\ell}{ \ell} \csc^2(  \pi \ell h/k) , \quad k \geq 2.
\]

 We need three additional functions $\Psi_{h,k}( x,w)$, $\omega_{h,k,n}( x)$, and $g_{h,k}(x,w)$ where
\begin{equation}\label{eq:Psi}
\Psi_{h,k}( x,w) = \frac{   Li_3(x^k)}{ k^2 w} ,
\end{equation}
\begin{equation}\label{eq:omega}
\ln[ \omega_{h,k,n}( x) ]
=
\begin{cases}
\frac{1}{12k} \ln( 1-x^k) - A_{h,k}(x, 0) - 2 \pi i n h/k , & k \geq 2
\\
\frac{1}{12} \ln( 1-x), & k=1 ,
\end{cases}
\end{equation}

\begin{equation}\label{eq:g_hk}
g_{h,k}(x,w) = 
\begin{cases}
\left[ A_{h,k}(x, w) - A_{h,k}(x,0) \right]
+
\left[ B_{h,k}( x,w ) - \Psi_{h,k}(x,w) - \frac{1}{12k} \ln (1-x^k) \right] , & k \geq 2
\\
B_{0,1}(x,w) - \Psi_{0,1}(x,w) - \frac{1}{12} \ln(1 -x), & k=1 .
\end{cases}
\end{equation}

By construction, $P(x,e^{-w + 2 \pi i h/k})$ admits a factorization as follows.
\begin{proposition} \label{prop:factorization}
Let $h,k, n$ be nonnegative integers such that $(h,k)=1$. If $ \Re w >0$ and $x \in {\mathbb D}$,
then the generating function $P(x, e^{ - w + 2 \pi i h/k})$ factors as
\begin{align*}
P(x, e^{ - w + 2 \pi i h/k})
&=
\omega_{h,k,n}(x) e^{ 2 \pi  i n  h/k} e^{\Psi_{h,k}(x,w)} e^{g_{h,k}(x,w)   },
\end{align*}
where $\omega_{h,k,n}(x)$, $\Psi_{h,k}(x,w)$ and $g_{h,k}(x,w)$ are given in  equations
(\ref{eq:Psi}), (\ref{eq:omega}),  (\ref{eq:g_hk}).
\end{proposition}

The key in using this factorization for the circle method  is the following bound.
\begin{proposition}\label{prop:g_bound}
Assume that $0<|x|<1$ and $\Re w>0$.
(a)
If $\Im w \neq 0$ as well, then
there exists $M>0$ such that
\begin{align}  \label{eq:g_hk_inequality}
 | g_{h,k}(x,w)  |
&  \leq
\frac{  2 |w|}{ 1- |x|     } 
\left[
k^3 + \frac{  |x|^{ \frac{ \pi}{ k | \Im w|}      }    }{  1- e^{ - \frac{ \pi \Re w}{k | \Im w|}     } }
\right]
\\ \nonumber
 &+ \quad
\frac{1}{ 1- |x|^2}
 \left[
 M |w|^2 k +
\left(
\frac{2}{  (1- e^{   - \Re w \frac{ \pi}{ |w|}          } )^3        }        +1 
\right)
  |x|^{ \frac{ \pi }{ |w|}       }
\right]   .
\end{align}
(b)
If $w$ is real and positive,   then
there exists $M>0$ such that
\begin{align*} 
 | g_{h,k}(x,w)  |
&  \leq
\frac{  2 |w|}{ 1- |x|     } 
k^3 
 +
\frac{1}{ 1- |x|^2}
 \left[
 M |w|^2 k +
\left(
\frac{2}{  (1- e^{   - \Re w \frac{ \pi}{ |w|}          } )^3        }        +1 
\right)
  |x|^{ \frac{ \pi }{ |w|}       }
\right]   .
\end{align*}
\end{proposition}
The proof is given in the following two lemmas.

\begin{lemma}\label{lemma:A_bound}
Let $k \geq 2$ and $\Re w>0$.
(a)
 If  $\Im w \neq 0$, then
\[
| A_{h,k}(x,w) - A_{h,k}(x,0) | 
\leq \frac{  2 |w|}{ 1- |x|     } 
\left[
k^3 + \frac{  |x|^{ \frac{ \pi}{ k | \Im w|}      }    }{  1- e^{ - \frac{ \pi \Re w}{k | \Im w|}     } }
\right].
\]
(b)
When $\Im w=0$, 
$\displaystyle
| A_{h,k}(x,w) - A_{h,k}(x,0) | 
\leq \frac{  2 |w|}{ 1- |x|     } k^3 .
$
\end{lemma}
\begin{proof}
(a)
For fixed $w$ and $x$, consider the function of $t$, $A_{h,k}(x,w t)$. 
By the Mean Value Theorem, we find
\begin{align*}
| A_{h,k}(x,w) - A_{h,k}(x,0) | 
&\leq
\sup_{0 <  t  <1}\left| \frac{d}{dt}  A_{h,k}(x,wt) \right|
\\
&\leq
|w| \,
\left|  \,
\sum_{k \nmid \ell} x^\ell
 \frac{    e^{ - \ell w t + 2 \pi i \ell h/k}      (1+e^{ - \ell w  t + 2 \pi i \ell h/k}  )         }  
 {     (1- e^{ - \ell w t + 2 \pi i \ell h/k}  )^3      } \,
 \right|
 \\
 &\leq
 2 \, |w| \, \sum_{ k \nmid \ell} \frac{ |x|^\ell }{  |  e^{ - w t \ell + 2 \pi i \ell h/k} -1|^3}  .
\end{align*}
For $0<t<1$ and $\ell \nmid k$ such that $ \ell | \Im w | t < \pi/k$, we have the bound
\[
\frac{ 1 }{  |  e^{ - w t \ell + 2 \pi i \ell h/k} -1|} 
\leq
\frac{1}{ |  \sin( 2 \pi \ell h/k - \Im w t \ell) |} \leq  | \csc( \pi/k)| \leq k/2
\]
since $ | \sin(y)| \leq | e^{x+iy} -1|$.
On the other hand, if $ \ell \geq \frac{ \pi }{  k | \Im w| t}$, we see that
\[
\frac{ 1 }{  |  e^{ - w t \ell  t+ 2 \pi i \ell h/k} -1|}  
\leq
 \frac{1}{  1- e^{ - \Re w t \ell}}    \leq \frac{1}{ 1 - e^{ - \frac{ \pi \Re w}{k | \Im w|} }       } .
\]
Combining these last two bounds, we can complete the proof:
\begin{align*}
\left|  \frac{d}{dt} A_{h,k}(x,wt) \right|
&\leq
2 \sum_{ k \nmid \ell} \frac{ |x|^\ell }{  |  e^{ - w t \ell + 2 \pi i \ell h/k} -1|^3  } 
\\
&\leq
\sum_{  { k \nmid \ell,} \atop{ \ell < \pi/[k | \Im w|} t] } 2 k^3 |x|^\ell
+
 \frac{ 2 }{  (1-  e^{ -  \frac{ \pi \Re w}{ k | \Im w|}  })^3      } 
 \sum_{  { k \nmid \ell,} \atop{ \ell \geq \pi/[k | \Im w|  t ] }  }   |x|^\ell
 \\
 &\leq
 2k^3 \, \frac{1}{1- |x|} +  \frac{ 2 }{  (1-  e^{ -  \frac{ \pi \Re w}{ k | \Im w|}  })^3      } \, 
 \frac{  |x|^{ \frac{ \pi}{ k   |\Im w| t}} }{  1- |x| }
 \\
 &\leq
 \frac{2}{ 1- |x|} \left[
 k^3 + \frac{  |x|^{ \frac{ \pi}{ k | \Im w|}      }    }{  1- e^{ - \frac{ \pi \Re w}{k | \Im w|}     } }
 \right]  .
\end{align*}
(b)
Both $|x|^{ \frac{ \pi}{ k | \Im w|}}$ and $e^{ - \pi \frac{ \Re w}{ k | \Im w|}}$ go to 0 as $\Im w \to 0$, so 
the bound in part (b) follows from part (a).
\end{proof}

\begin{lemma}\label{lemma:B_bound}
  Assume that $0<|x|<1$ and $\Re w>0$. There exists $M>0$ such that
\begin{equation}\label{eq:B_bound}
\left| B_k(x,w) - \frac{ Li_3(x^k)}{ k^3 w^2} - \frac{1}{12} \ln(1-x^k) \right|
\leq
\frac{1}{ (1- |x|)^2} \left[
M |w|^2 k + 
\left(
\frac{2}{ ( 1- e^{ - \Re w \frac{ \pi}{ |w|}})^3} +1
\right)  |x|^{ \frac{ \pi}{ |w|}}
\right]  .
\end{equation}
\end{lemma}
\begin{proof}
We begin by expanding the left-hand side of (\ref{eq:B_bound}) as a series
\[
\sum_{\ell=1}^\infty \frac{ x^{ k \ell}}{ k \ell}
\left[
\frac{ e^{ - wk \ell}}{ ( e^{ - wk \ell} -1)^2}  - \frac{1}{ \ell^2 k^2 w^2} + \frac{1}{12}
\right]  .
\]
For $| kw \ell | < \pi$, there exists $M>0$ such that
\[
\left|
\frac{ e^{ - wk \ell}}{ ( e^{ - wk \ell} -1)^2}  - \frac{1}{ \ell^2 k^2 w^2} + \frac{1}{12}
\right|
\leq M | k w \ell |^2 
\]
since 
\[
\lim_{z\to 0} \frac{1}{z^2} \left(  \frac{ e^z}{ (e^z-1)^2} - \frac{1}{z^2} + \frac{1}{12} \right) = \frac{1}{240} .
\]
On the other hand, for $| kw \ell | > \pi$,  
\[
\left|
\frac{ e^{ - wk \ell}}{ ( e^{ - wk \ell} -1)^2}
\right| \leq
\frac{ e^{ - \Re w k \ell}}{  ( e^{ - \Re w k \ell} -1)^2} \leq \frac{2}{ ( 1- e^{ - \Re w \frac{ \pi }{ |w|}})^3}  .
\]
So we have the bound for $\Re w>0$
\[
\left|
\frac{ e^{ - wk \ell}}{ ( e^{ - wk \ell} -1)^2}  - \frac{1}{ \ell^2 k^2 w^2} + \frac{1}{12}
\right|
\leq
\frac{2}{ (1- e^{ - \Re w \frac{ \pi}{ |w|}})^3} +1
\]
The proof is  now completed in the same way as Lemma \ref{lemma:A_bound}.
\end{proof}

\section{Phases}

For convenience, we record a result from our paper \cite{Boyer_Parry_phase}.

\begin{definition} \cite[Definition 1]{Boyer_Parry_phase}
\label{def:phase}
 Let $\{ L_k(x) \}$ be any sequence of complex-valued  functions on a domain $D$.  
The  set $R(m)$  is the  $m$-th phase (or phase $m$) of $\{ L_k(x)\}$
if
(1) if $x \in R(m)$, then $\Re L_m(x) > \Re L_k(x)$ for all $k \neq m$ and
(2)
 if $V$ is any open subset of $D$ satisfying (1), then $V \subset R(m)$.
\end{definition}

\begin{theorem}(Parry-Boyer \cite{Boyer_Parry_phase})\label{thm:phase}
Let $D$ be the punctured open unit disk and let $\{ L_k(x) \}$ be given as in (\ref{def:L_k}). 
  Then $D$ contains  exactly
two nonempty phases $R(1)$ and $R(2)$ whose union is dense in $D$ and
whose common boundary is  the level set
 $\{x\in D: \Re L_1(x)=\Re L_2(x)  \}.$  
 It has exactly one real  point $x^* \simeq  -0.82500 \, 30529 $ 
and its closure contains exactly  two points $e^{ \pm  i \theta^* }$ on the unit circle   where  
$\theta^* \simeq \pm 0.95170 \, 31251  \pi.$
Further $R(2)$ lies in the open left half plane and  $R(2) \cap {\mathbb R} = (-1,x^*)$.

\end{theorem}

\section{Asymptotics of the Polynomials on the Phases $R(1)$ and $R(2)$ }\label{section:asymptotics}

\subsection{Introduction}
We will adapt the circle method which is usually used to give asymptotics for a sequence $\{c_n\}$ of positive numbers through their generating
function $\sum_{n\geq 0} c_n u^n$.   With the  coefficients $c_n$  replaced with the polynomials $Q_n(x)$, we find that the dominant contribution
to their asymptotics depends on the location of $x$ in the unit disk ${\mathbb D}$.  The purpose of this section is to show
that the  subsets of $\mathbb D$ where the asymptotics have the same
dominant form coincide with the phases of $\{ L_k(x)\}$.
See Andrews's classic book \cite[Chapter 5]{Andrews_book}  for a thorough discussion of the circle method.

   \begin{theorem} \label{thm:Main1}
   Let $R(1)$ and $R(2)$ be the phases of $\{ L_k(x) \}$ given in Theorem \ref{thm:phase}.
   \newline
 (a)
Let  $x\in X\subset R(1) \setminus [x^*,0]$ be a compact set,  then
\[
Q_n(x)= 
\sqrt[12]{1-x}  \sqrt{\frac{L_1(x)} { 6\pi n^{4/3} }}
\exp \left(  \tfrac{3} {2} n^{2/3} L_1(x) \right)
\,
\left( 1+O_X\left(  n^{-1/3} \right)  \right) .
\]
\noindent
(b)
Let  $x\in X\subset R(2)$ be a compact, then 
\begin{eqnarray*}
Q_n(x)
=
(-1)^n\sqrt[24]{1-x^2}
\,
\sqrt[8]{\frac{1-x}{1+x}} 
\,
 \sqrt{\frac{L_2(x)} { 6\pi n^{4/3} }  }
 \,
 \exp \left(  \tfrac{3} {2} n^{2/3} L_2(x) \right)
\left(1+O_X\left(  n^{-1/3} \right)  \right).
\end{eqnarray*}
\end{theorem}

By the Cauchy integral formula, we have an integral expression for $Q_n(x)$
with integration contour the circle $| u|= e^{ - 2\pi \alpha}$ where 
\begin{equation}\label{eq:alpha}
\alpha = \frac{1}{ 2 \pi n^{1/3}} \Re L_m(x)
\end{equation}
and $m =1$ or $2$.
Then we dissect the circular contour relative to a Farey sequence $F_N$ of order $N = \lfloor \delta n^{1/3} \rfloor$ as
follows 
\begin{align*}
Q_n(x)
&=
\frac{1}{2 \pi i } \, \oint_{  |u| = e^{ - 2 \pi \alpha}} \frac{ P(x,u) \, du}{ u^{n+1}}
\\
&=
\sum_{ h/k \in F_N} \, \int_{ h/k - 1/k(k+k')}^{ h/k + 1/ k( k+ k'')} 
\, \frac{P( x, e^{ 2 \pi ( - \alpha + i \psi)}  ) } { e^{ 2\pi n ( - \alpha + i \psi)}} \, 
d \psi
\\
&=
\sum_{ h/k \in F_N} \, \int_{ h/k - 1/k(k+k')}^{ h/k + 1/ k( k+ k'')} 
e^{ - 2 \pi i n h/k} e^{ 2 \pi n ( \alpha - i ( \psi -h/k))} 
P(s, e^{ - 2 \pi ( \alpha -i ( \psi - h/k)) + 2 \pi i h/k)})  \, d\psi
\end{align*}
where $h'/k' < h/k < h''/k''$ are consecutive elements from the Farey sequence $F_N$ in reduced form.
By convention, we will always
assume that $N > m$.

We now make the change of variables $v = \psi - h/k$ and apply the factorization in Proposition \ref{prop:factorization} to get
\[
P(x, e^{ - 2 \pi ( \alpha -i v) + 2 \pi i h/k)})
=
e^{ 2 \pi i n h/k} \omega_{h,k,n}(x)
\exp
\left(
\frac{ L_k(x)^3}{ 8 \pi^2 ( \alpha- iv)^2}  + g_{h,k}(x, 2 \pi ( \alpha - iv))  .
\right)
\]
We introduce the integral
\begin{equation}
I_{h,k,n}(x)
=
 \int_{ h/k - 1/k(k+k')}^{ h/k + 1/ k( k+ k'')} 
 \exp
\left(
\frac{ L_k(x)^3}{ 8 \pi^2 ( \alpha- iv)^2}  + 2 \pi n ( \alpha - iv)
\right) 
\
e^{ g_{h,k}(x, 2 \pi ( \alpha - iv))}
\, dv .
\end{equation}
So we may write $Q_n(x)$ as
\begin{equation}\label{eq:Q_n(x)_sum}
Q_n(x) = \sum_{ h/k\in F_N} \omega_{h,k,n}(x) I_{h,k,n}(x)  .
\end{equation}

Our next goal is to  show that $\omega_{1,m,n}(x) I_{1,m,n}(x)$ is the dominant term in this expansion.

\subsection{Major arcs}\label{subsection:major_arcs}

We start by making another change of variables $z= 2 \pi n^{1/3} v$ in $I_{h,k,n}(x)$ to get
\begin{align*}
I_{h,k,n}(x) 
= 
\frac{1}{2 \pi n^{1/3}} \,
\int_{    - \frac{2 \pi n^{1/3}   }{ m+m' }  }^{ \frac{2 \pi n^{1/3}}{ m+m''}  }
\,
&
\exp
\left[
n^{2/3} 
\left(
\frac{ L_m(x)^3}{ 2 ( \Re L_m(x) - iz)^2 } + ( \Re L_m(x) - iz)
\right)
\right] \, 
\\
& \qquad
\times
\exp
\left[
g_{h,k} \left( x,  \frac{1}{ n^{1/3}} ( \Re L_m(x) - iz)  \right)
\right]
\,dz  .
\end{align*}

We decompose $I_{h,k,n}(x)$ into the sum $I'_{h,k,n}(x) + I''_{h,k,n}(x)$ where
\[
I'_{h,k,n}(x) = 
\frac{1}{2 \pi n^{1/3}} \,
\int_{    - \frac{2 \pi n^{1/3}   }{ m+m' }  }^{ \frac{2 \pi n^{1/3}}{ m+m''}  }
\,
\exp
\left[
n^{2/3} 
\left(
\frac{ L_m(x)^3}{ 2 ( \Re L_m(x) - iz)^2 } + ( \Re L_m(x) - iz)
\right)
\right] \, 
dz  .
\]
and
\begin{align*}
I''_{h,k,n}(x) 
=
\frac{1}{2 \pi n^{1/3}} \,
\int_{    - \frac{2 \pi n^{1/3}   }{ m+m' }  }^{ \frac{2 \pi n^{1/3}}{ m+m''}  }
&\exp
\left[
n^{2/3} 
\left(
\frac{ L_m(x)^3}{ 2 ( \Re L_m(x) - iz)^2 } + ( \Re L_m(x) - iz)
\right)
\right] \, 
\\ 
\qquad \qquad \qquad
&\times \quad
\left\{
\exp
\left[
g_{h,k} \left( x,  \frac{1}{ n^{1/3}} ( \Re L_m(x) - iz)  \right)
\right]
-1
\right\}
\,dz   .
\end{align*}

We will obtain an asymptotic expansion of $I'_{h,m,n}(x) $ using a saddle point expansion. To begin we need a basic
inequality.

\begin{lemma}\label{lemma:basic_inequality}
Let  $\alpha >0$, $v \in {\mathbb R}$, and $| \arg(L) |\leq \pi/3$, then
\[
\Re
\left(
\frac{ L^3}{ ( \alpha - iv)^2} \right)
\leq \frac{ ( \Re L)^3}{ \alpha^2} .
\]
If $|\arg(L)| < \pi/3$ and $L \neq 0$, then equality is attained uniquely at $v= - \Im L$.
If $| \arg(L) | = \pi/3$ and $L \neq 0$, then equality is attained only at $v = \pm \Im L$.
\end{lemma}

\begin{lemma}
Let $x\in X\subset D$ be compact.
For $0 < \delta < \pi/ \sqrt[3]{ 2 \zeta(3)}$, then 
\newline
(a) for every $x\in X$, 
$\pm \Im L_m(x) \in ( - \pi/ m \delta, \pi/ m \delta) 
\subset [  - \frac{2 \pi n^{1/3}   }{ m+m' } ,\frac{2 \pi n^{1/3}}{ m+m''} ]$.
\newline
(b)
there exists a constant $K>0$ such that
\begin{align*}
I'_{h,m,n}(x)
&=
\frac{1}{2 \pi n^{1/3}} \,
\int_{    - \frac{ \pi}{ m \delta}  }^{ \frac{ \pi}{ m \delta}  }
\,
\exp
\left[
n^{2/3} 
\left(
\frac{ L_m(x)^3}{ 2 ( \Re L_m(x) - iz)^2 - iz)^2} + ( \Re L_m(x) - iz)
\right)
\right] \, dz
\\
&
+\quad
O
\left(
\exp\left(
n^{2/3} [  \tfrac{3}{2} \Re L_m(x) - K  ]
\right)
\right)  .
\end{align*}
\end{lemma}
\begin{proof}
(a)
Since the Farey sequence has order $\lfloor \delta n^{1/3} \rfloor$, we have
\[
\frac{ \pi}{ m \delta} \leq \frac{ 2 \pi n^{1/3}}{ m+m'} , \quad
\frac{ 2 \pi n^{1/3}}{ m+m''} \leq \frac{ 2 \pi n^{1/3}}{ m \lfloor \delta n^{1/3} \rfloor}
\]
since $h'/m' < h/m < h''/m''$ are consecutive terms of the Farey sequence. In particular,
with $0<\delta < \pi/ \sqrt[3]{ 2 \zeta(3)}$, we find $| \Im L_m(x)|  \leq \frac{1}{m} \sqrt[3]{ 2 \zeta(3)} < 
\pi/m \delta$.

\noindent
(b)
By Lemma  \ref{lemma:basic_inequality}, if $ |v| \geq \pi/m \delta$, there exists $K>0$ such that for all $x\in X$
\[
\Re\left(
\frac{ L_m(x)^3}{ 2 ( \Re L_m(x) - iz)^2 - iz)^2} + ( \Re L_m(x) - iz)
\right)
\leq \frac{3}{2} \Re L_m(x) - K .
\]
We consider the estimates for the integral on the set
$J = 
[ - \frac{2 \pi n^{1/3}   }{ m+m' }  ,  \frac{2 \pi n^{1/3}}{ m+m''} ]  
  \setminus   ( - \frac{ \pi}{ m\delta}  ,   \frac{ \pi}{ m\delta} )$:
\begin{eqnarray*}
\lefteqn
{
\left|
\int_{  J }
\,
\exp
\left[
n^{2/3} 
\left(
\frac{ L_m(x)^3}{ 2 ( \Re L_m(x) - iz)^2 } + ( \Re L_m(x) - iz)
\right)
\right] \, dz
\right|
}
\\
&&
\qquad \qquad\qquad
=\quad
\int_{   J  }
\,
\exp
\left[
n^{2/3} 
\Re \left(
\frac{ L_m(x)^3}{ 2 ( \Re L_m(x) - iz)^2 } + ( \Re L_m(x) - iz)
\right)
\right] \, dz
\\
&&
\qquad \qquad\qquad
\leq \quad
\int_{   J }
\,
\exp
\left[
n^{2/3} 
 \left(
\frac{3}{2} \Re L_m(x) - K 
\right)
\right] \, dz
\\
&&
\qquad \qquad\qquad
= \quad
O\left(
n^{1/3}
\exp
\left[
n^{2/3} 
 \left(
\frac{3}{2} \Re L_m(x) - K 
\right)
\right] 
\right)  .
\end{eqnarray*}
The estimate in part (b) now follows.
\end{proof}

\begin{proposition}\label{prop:I_hm_saddle}
(a)
Let $x \in X \subset {\mathbb D} \setminus \{ x^m \leq 0\}$ be compact. Then 
\[
I'_{h,m,n}(x)
=
\frac{ 1}{ \sqrt{ 2 \pi n^{2/3}} } \, \sqrt{ \frac{ L_m(x) }{3}} \, \exp\left( \tfrac{3}{2} n^{2/3} L_m(x)\right)
\,
 \left(  1+ O_X( n^{ - 2/3}) \right) .
\]
(b)
Let $x \in X \subset \{ x \in D : x^m \leq 0\}$ be compact. Then we have
\begin{align*}
I'_{h,m,n}(x)
=
\frac{1}{  \sqrt{ 2 \pi n^{2/3}}} \, 
&\left[
\sqrt{\frac{ L_m(x)}{3} } \,e^{ \tfrac{3}{2} n^{2/3} L_m(x)} \right] \left( 1+O_X( n^{-2/3}) \right)
\\
& \qquad
+ 
\frac{1}{ \sqrt{ 2 \pi n^{2/3}}} \, \left[
\sqrt{\frac{ L_m(x)}{3}} \, e^{ \tfrac{3}{2} n^{2/3} L_m(x)} \right]^{-} \left( 1+O_X( n^{-2/3}) \right)  .
\end{align*}
\end{proposition}
\begin{proof}
We apply the saddle point method as given in 
\cite[p. 10-11]{Pinsky} and  note that standard arguments will make the estimates there uniform
for $x\in X$.
We let $B(z)=\frac{ L_m(x)^3}{ 2 ( \Re L_m(x) - iz)^2 } + ( \Re L_m(x) - iz)$.

(a)
By the inequality in Lemma \ref{lemma:basic_inequality},   $B(z)$
  has a unique maximum on $ {\mathbb R}$
at $z_0= - \Im L_m(x)$
with
 $B(z_0)= \frac{3}{2 \Re L_m(x)}>0$.
Note that uniqueness follows since, for $x\in X$,  $Li_3(x) <0$ if and only if $x<0$.
Now the second derivative $B''(z)=-3/( \Re L_m(x) - iz)^4$ is Lipschitz continuous on
$\mathbb R$ and
$B''(z_0)= - 3/ L_m(x)$.
Hence, by \cite[p. 11]{Pinsky}, we find that as $n\to\infty$
\begin{eqnarray*}
\int_{    - \frac{ \pi}{ m \delta}  }^{ \frac{ \pi}{ m \delta}  }
e^{ n^{2/3} B(z)} \, dz
=
e^{ n^{2/3} B(z_0)} \, \sqrt{   \frac{ \pi}{ n^{2/3} \cdot \frac{3}{2} \, \frac{1}{ L_m(x)}    }}
\,
\left( 1+O_X( n^{-2/3}) \right) 
\end{eqnarray*}
which simplifies to the desired expression.

\noindent
(b)
Let $z_0= \Im L_m(x)$.
Since $x^m <0$, for $z \in {\mathbb R}$, $B(z)$ has two equal maxima at $\pm z_0$ with 
$B(\pm z_0)=  \frac{3}{2} L_m(x)$. 
Since $|\arg L_m(x)| = \pi/3$ and $x\neq 0$,  $\Im L_m( x) \neq 0$.  
Now $B(-z) = \overline{ B(z)}$ for $z \in {\mathbb R}$ since $L^3$ is real.  
We get desired the equation by applying the saddle point method as in part (a) to
$\displaystyle 
\int_0^{\frac{ \pi}{m \delta}} e^{ n^{2/3} B(z)} \, dz$ and observing that
$\displaystyle
\left[ \int_0^{\frac{ \pi}{m \delta}} e^{ n^{2/3} B(z)}   \, dz \right]^{-}= \int_{- \frac{ \pi}{ m \delta}}^0 e^{n^{2/3}B(z)}\, dz$.
\end{proof}

Note that bounds that come from the above saddle point asymptotics must exclude the interval $[x^*,0]$ from the phase
$R(1)$. This explains choosing between $R(1)$ and $R(1) \setminus [x^*,0]$ as a region below.

\subsection{Estimates for Non-dominant Contributions}

\subsubsection{Bounds for the Integrals $I_{h,m,n}''(x)$}

\begin{lemma}\label{lemma:g_estimate}
Let $x\in X \subset D$ be compact with  $M_X = \max\{ |x| : x \in X\}$. Fix $\delta$ so
$0< \delta < 1/\sqrt[3]{2 \zeta(3)}$ and $n \geq 1/\delta^3$. Then
\newline
(a)
For $k\in {\mathbb N}$ fixed, then as $n\to\infty$
\[
\left|
g_{h,k} \left( x,  \frac{1}{ n^{1/3}} ( \Re L_m(x) - iz)  \right)
\right| = O_{X,k} \left( \frac{1}{ \delta n^{1/3} } \right) .
\]
(b)
For $k \leq \lfloor \delta n^{1/3} \rfloor$, then as $n\to\infty$
\[
\left|
g_{h,k} \left( x,  \frac{1}{ n^{1/3}} ( \Re L_m(x) - iz)  \right)
\right| 
\leq
\frac{2 \sqrt{5}}{ 1-M_X} \delta^2 n^{2/3} + O_{X} \left( \frac{1}{ \delta n^{1/3  }} \right) .
\]
\end{lemma}
\begin{proof}
We will bound the terms in equation (\ref{eq:g_hk_inequality})  individually. For convenience, we 
work with the variable $v= 2 \pi n^{1/3} z$. 
Recall that
$\sqrt[3]{2 \zeta(3)}< 1/\delta$
and
 $|v| \leq  \frac{1}{k  N   }  <\frac{1}{ k \delta n^{1/3}}$,
where $N= \lfloor  \delta n^{1/3} \rfloor$ is the order of the Farey fractions.
Introduce 
$C_X = \max \{ 1/ \Re L_m(x) : x\in X\}$. 
We start with the
easy estimate:  
\[
|w| = \left| \frac{\Re L_m(x)}{ n^{1/3}} -  iv  \right| \leq
\frac{1}{n^{1/3}} \sqrt{  (2 \zeta(3))^{2/3} + \frac{1}{ (k \delta)^2}  }
\leq
 \frac{ \sqrt{5}}{ \delta n^{1/3}} 
\]
since $\sqrt[3]{2 \zeta(3)}$ is an upper bound for $\Re L_m(x)$, $x\in X$.

Next by using that $1/(1-e^{-t})< 1+ 1/t$ for $t>0$, we can show that
\begin{align*}
 \frac{1}{ 1- e^{ - \frac{\pi \Re w}{ k | \Im w|}}    }
 \leq  1+ \frac{ \sqrt{5} C_X}{ \pi \delta} ,
 \qquad
  \frac{1}{ 1 - e^{  - \frac{ \pi \Re w}{ |w|}}  } 
 \leq
 1+ \frac{ \sqrt{5} C_X}{ \pi \delta}
\end{align*}
Using $e^{-t}\leq 27/(et)^3$ for $t>0$, we obtain
\begin{align*}
 |x|^{ \frac{ \pi}{ k | \Im w  | } } 
 \leq
  \frac{ - 27}{ ( 2 \pi e n^{1/3} \delta \ln M_X)^3} ,
 \qquad
   |x|^{ \frac{ \pi}{ |w|}  } 
   \leq 
     \frac{ - 27 \sqrt{125}}{ (\pi  e\delta n^{1/3} \ln M_X)^3   }   .
\end{align*}
Hence, $\left|
g_{h,k} \left( x,   \frac{\Re L_m(x)}{ n^{1/3}} - iv  \right)
\right|$ 
is bounded above by
\begin{align*}
& \frac{ 2  \sqrt{5}}  { \delta n^{1/3} (1- M_X)} 
  \left[
 k^3 + \frac{ -27}{ ( 2 \pi e \delta n^{1/3} \ln M_X)^3  } 
   \cdot 
 \left( 1+ \frac{ 2 C_X}{ \pi \delta} \right)
 \right]
 \\
 & \qquad  +
 \frac{1}{ ( 1 - M_X)^2} 
 \left[
\frac{5 M  k}{  \delta^2 n^{2/3}} 
  +
 \left(
 2 \cdot 
 \left(  1+ \frac{ \sqrt{5} C_X}{ \pi \delta} \right)^3  +  1
 \right)
 \cdot
 \frac{ - 27 }{ ( \pi e \delta n^{1/3} \ln M_X)^3}
 \right]  .
 \end{align*}
We rewrite this upper bound as a polynomial in $k$ with constants $A$ and $B$ independent of
$\delta, k$, and $n$:
\begin{equation}\label{eq:lemma2}
\frac{ 2 \sqrt{5}  }{ \delta n^{1/3} (1-M_X)} k^3 + 
 \frac{1}{ ( 1 - M_X)^2} \, \frac{ 5M }{ \delta^2 n^{2/3}  } k +A \, \frac{1}{ \delta^5 n^{4/3}} + B \, \frac{1}{ \delta^6 n} .
\end{equation}
 With $k$ fixed, we obtain part (a) since the bound in (\ref{eq:lemma2}) is $O_X( 1/[ \delta n^{1/3} ])$; while if we
 replace
 $k$ with $\delta n^{1/3}$, we get the bound in part (b).
 \end{proof}

\begin{lemma}\label{lemma:I''_estimate}
Let $x \in X$ where $X$ is a compact subset of either $R(1) \setminus [x^*,0]$ or $R(2)$. Then
$I_{h,m,n}''(x) = O_{X,\delta}( I_{h,m,n}'(x)  / n^{1/3} )$, $m=1,2$.
\end{lemma}
\begin{proof}
By the definition of $I_{h,m,n}''(x)$ given in subsection \ref{subsection:major_arcs},
it is enough to
 observe that there exists a positive constant $K_{X,m,\delta}$ such that
\[
\left|
e^{ g_{h,m} \left( x,  \frac{1}{ n^{1/3}} ( \Re L_m(x) - iz)  \right)       } 
\right|
\leq \frac{K_{X, m,\delta}}{ n^{1/3}} 
\]
by the above lemma.

\end{proof}

\subsubsection{Bounds for Minor Arcs}

\begin{lemma}\label{lemma:bound_omega}
Let $x \in X$ where $X$ is a compact subset of unit disk. Then
\[
|  \omega_{h,k,n}(x) |
\leq
2^{1/12}\exp \left( \frac{k^2}{ 16h^2} Li_3( M_X) \right) 
\leq 
2^{1/12} \exp \left( \frac{k^2  }{ 16}\, ( \zeta(3) - \ln(1-M_X) ) \right) 
\]
where $M_X = \max\{ |x| : x\in X\}$ and $\omega_{h,k,n}(x)$ is given in (\ref{eq:omega}).
\end{lemma}
\begin{proof}
We bound each component separately in the definition of $\omega_{h,k,n}(x)$.
We start with the first factor
\begin{align*}
 \left|  \exp\left( \frac{1}{12k} \ln(1-x^k) \right) \right| = \exp \left( \frac{1}{12k}\Re \ln(1-x^k) \right)
 \leq 2^{1/12k} .
\end{align*}
Next we obtain an intermediate bound
\begin{align*}
\left| \exp \left(- \frac{1}{4} \sum_{ \ell \nmid  k} \frac{ x^\ell}{ \ell} \csc^2( \pi h \ell /k) \right) \right|
\leq
\exp \left( \frac{1}{4} \sum_{ \ell \nmid  k}  \frac{ |x|^\ell}{ \ell}  \,  \csc^2( \pi h \ell /k) \right)   .
\end{align*}
Break up the above series into two sums with the sets of indices 
\[
J_1 = \{ \ell : \ell \nmid k, \,  h \ell/k \bmod 1 < 1/2\},
\quad
J_2 = \{  \ell \nmid k, \, 1/2 \geq  h \ell/k \bmod 1 \}.
\]
and
 use the elementary estimates  $\sin \theta \geq 2 \theta/\pi$, for $0\leq \theta \leq \pi/2$,
and $\geq 2(1- \theta/\pi)$ for $\pi/2 \leq \theta \leq \pi$
to obtain the two bounds below to complete the proof:
\begin{align*}
\exp \left( \frac{1}{4} \sum_{ \ell \in J_1}  \frac{ |x|^\ell}{ \ell}  \,  \csc^2( \pi h \ell /k) \right) 
&\leq
\exp \left( \frac{k^2}{  16h^2 } Li_3( M_X) \right)  ,
\\
\exp \left( \frac{1}{4} \sum_{ \ell \in J_2}  \frac{ |x|^\ell}{ \ell}  \,  \csc^2( \pi h \ell /k) \right) 
&\leq
\exp \left( -\frac{k^2}{16} \ln(1-M_X)    \right)  .
\end{align*}
\end{proof}

\begin{lemma}
Let $x \in X$ where $X$ is a compact subset of $R(m)$, $m=1,2$.
There exists a positive constant $a_{X}$ such that
\[
 \Re
\left[
\frac{  L_k(x)^3}{ 8 \pi^2 ( \alpha - iv)^2} + 2 \pi n ( \alpha - iv)
\right]
\leq  
n^{2/3}
\left(
\frac{3}{2} \Re L_m(x) 
+
a_{X}
\right), \quad k \neq m, \, x \in X ,
\]
where $\alpha$ is given in equation (\ref{eq:alpha}).
\end{lemma}
\begin{proof}
In subsection \ref{subsection:major_arcs}, we saw that
\begin{align*}
\frac{1}{n^{2/3}} 
\Re  &
\left[
\frac{  L_k(x)^3}{ 8 \pi^2 ( \alpha - iv)^2} + 2 \pi n ( \alpha - iv)
\right] 
\\
& \qquad \qquad \quad
=
\Re
\left[
\frac{ L_k(x)^3}{ 2 ( \Re L_m(x) - 2 \pi i n^{1/3} v )^2}
+
( \Re L_m(x) - 2 \pi i n^{1/3} v )
\right]  .
\end{align*}
By Lemma \ref{lemma:basic_inequality}, we have the strict bound for  $x\in X$ and $k \neq m$
\begin{align*}
\Re
\left[
\frac{ L_k(x)^3}{ 2 ( \Re L_m(x) - 2 \pi i n^{1/3} v )^2}
\right.
&+
\left.
( \Re L_m(x) - 2 \pi i n^{1/3} v )
\right]
\\
&  \qquad 
\leq 
\frac{ ( \Re L_k(x) )^3}{ 2 ( \Re L_m(x))^2} + \Re L_m(x)
< 
\frac{3}{2} \Re L_m(x) .
\end{align*}
Hence  the difference, for $k \neq m$, has  a positive minimum $a_{X,k}$ on $X$:
\[
\frac{3}{2} \Re L_m(x)
-
\Re
\left[
\frac{ L_k(x)^3}{ 2 ( \Re L_m(x) - 2 \pi i n^{1/3} v )^2}
+
( \Re L_m(x) - 2 \pi i n^{1/3} v )
\right] \geq a_{X,k} >0, \, x \in X   ,
\]
by compactness. 
Consider $a_{X,k}$ as a sequence, then it converges to the minimum of $\frac{3}{2} \Re L_m(x)$ on $X$
which is positive.  
Hence, 
$\inf \{ a_{X,k} : k \neq m\}$ is attained for some index, say $k_0 \neq m$. In particular, we have
\[
a_{X,k} \geq a_{X,k_0} > 0, \quad k \neq m .
\]
For simplicity, we write $a_X$ for $a_{X,k_0}$. The inequality in the lemma follows. 
\end{proof}

\subsection{Conclusion of proof of Theorem \ref{thm:Main1} }

\begin{lemma}\label{lemma:delta0}
Let $x\in X$ be a compact subset  of $R(m)$, $m=1,2$. Then there exists $\delta_0>0$ and $\eta>0$ such that
\[
\left|
\sum_{ h/k \in F_N, k \neq m}
\omega_{h,k,n}(x) I_{h,k,n}(x)
\right|
\leq
2^{1/12}
\exp\left(   \, 
    (\tfrac{3}{2} \Re L_m(x)  - \eta) n^{2/3} + O_{X,\delta_0}(1)
\right)
\]

\end{lemma}

\begin{proof}
We need to estimate $I_{h,k,n}(x)$ where $k \neq m$:
\begin{align*}
I_{h,k,n}(x)
=
\omega_{h,k,n}(x) \,
\int_{- \frac{1}{ k (k+ k')}}^{\frac{1}{  k(k+k'')}}
\exp
\left(
\frac{  L_k(x)^3}{ 8 \pi^2 ( \alpha - iv)^2} + 2 \pi n ( \alpha - iv)
\right)
\exp\left(
g_{h,k}(x, 2 \pi ( \alpha - iv)) \right) \, dv .
\end{align*}
We begin with the bounds
\begin{eqnarray*}
\lefteqn
{
\left|
\int_{- \frac{1}{ k (k+ k')}}^{\frac{1}{  k(k+k'')}}
\exp
\left(
\frac{  L_k(x)^3}{ 8 \pi^2 ( \alpha - iv)^2} + 2 \pi n ( \alpha - iv)
\right)
\,
\exp\left(
g_{h,k}(x, 2 \pi ( \alpha - iv)) \right) \, dv 
\right|
}
\\
&&
\leq
\int_{- \frac{1}{ k (k+ k')}}^{\frac{1}{  k(k+k'')}}
\left|
\exp
\left(
\frac{  L_k(x)^3}{ 8 \pi^2 ( \alpha - iv)^2} + 2 \pi n ( \alpha - iv)
\right)
\right|
\,
\left|
\exp\left(
g_{h,k}(x, 2 \pi ( \alpha - iv)) \right)
\right|
 \, dv 
 \\
 &&\leq
 \int_{- \frac{1}{ k (k+ k')}}^{\frac{1}{  k(k+k'')}}
\,
\exp
\left(
- a_{X} n^{2/3} + \frac{3}{2} n^{2/3} \Re L_m(x) 
\right)
\,
\left|
\exp\left(
\frac{ 2 \sqrt{2}}{ 1 - M_X} \delta^2 n^{2/3} + o_{X,\delta}(1)
\right)
\right|
 \, dv   
\end{eqnarray*}
where we used Lemma \ref{lemma:g_estimate}.
Hence a  full bound for the sum over $h/k \in F_N, k \neq m$  is
\begin{eqnarray*}
\lefteqn
{
2^{1/12}
 \exp\left(
  - a_{X,\delta} n^{2/3} 
   + \frac{3}{2} n^{2/3} \Re L_m(x)
  + \frac{ 2 \sqrt{2}}{ 1 - M_X} \delta^2 n^{2/3} 
  \right.
  }
  \\
  &&
  \qquad\qquad
  \left.
  + \quad
  \delta^2 n^{2/3} \frac{ \zeta(3)}{16}  - \delta^2 n^{2/3}  \frac{ \ln(1-M_X)}{16}
  + O_{X,\delta}(1)
\right) 
\,
\left(
 \frac{1}{ k (k+ k')} + \frac{1}{  k(k+k'')}
\right)
 .
\end{eqnarray*}
where the last factor is  the  length of the interval of integration.
Note that this bound holds for any $h/k \in F_N$ with $k \neq m$.

Let $\delta_0>0$ be chosen so that
\begin{equation}\label{eq:delta0}
a_{X}- \frac{ 2 \sqrt{2}}{ 1 - M_X}\delta_0 -\frac{ \zeta(3)}{16}\delta_0 +  \frac{ \ln(1-M_X)}{16}\delta_0 > 0 .
\end{equation}
Set $\eta= a_{X}- \frac{ 2 \sqrt{2}}{ 1 - M_X}\delta_0 -\frac{ \zeta(3)}{16}\delta_0
 +  \frac{ \ln(1-M_X)}{16}\delta_0 >0$ so we now write    new full bound as
 \[
2^{1/12} \exp\left(   \, 
    (\tfrac{3}{2} \Re L_m(x)  - \eta) n^{2/3} + O_{X,\delta_0}(1)
\right).
\]
The bound for the contributions over all the minor arcs is
\[
\left|
\sum_{ h/k \in F_N, k \neq m}
\omega_{h,k,n}(x) I_{h,k,n}(x)
\right|
\leq
2^{1/12}
\exp\left(   \, 
    (\tfrac{3}{2} \Re L_m(x)  - \eta) n^{2/3} + O_{X,\delta_0}(1)
\right)
\]
since the sum of the lengths of all the minor arcs is less than $1$. The proof is complete.
\end{proof}

\begin{proof} 
We now complete the proof of Theorem \ref{thm:Main1}.
Choose the order $N$ of the Farey fractions to be $\lfloor \delta_0 n^{1/3} \rfloor$ where $\delta_0$ is given
in Lemma \ref{lemma:delta0}.
By equation (\ref{eq:Q_n(x)_sum}), we write
\[
Q_n(x) = 
\omega_{1,m,n}(x) I_{1,m,n}(x) + \sum_{ h/k \in F_N, k \neq m } \omega_{h,k,n}(x) I_{h,k,n}(x)
\]
By Proposition \ref{prop:I_hm_saddle} and Lemma \ref{lemma:I''_estimate},  we see that
\[
\omega_{1,m,n}(x) I_{1,m,n}(x) = \omega_{1,m,n}(x)  \, \sqrt{  \frac{ L_m(x)}{ 6 \pi n^{4/3} } }\exp(\tfrac{3}{2} n^{2/3} L_m(x)  ) 
\,  \left(1+O_X( n^{-1/3}) \right)
\]
which dominates the contribution of the sum over $h/k \in F_N, k \neq m$ by Lemma \ref{lemma:delta0}. 
Finally, we observe that
\[
\omega_{1,1,n}(x) = \sqrt[12]{1-x}, \quad
\omega_{1,2,n}(x) = (-1)^n \, \sqrt[24]{ 1-x^2} \, \sqrt[8]{ \frac{ 1-x}{1+x}} 
\]
to complete the proof.
\end{proof}

\section{ Asymptotics on the Boundaries of the Phases $R(1)$ and $R(2)$}

Let $x \in X \subset \{ x : \Re L_1(x) = \Re L_2(x), \, x \neq x^* \}$ be compact. Then the choice of $\alpha$ in equation (\ref{eq:alpha})  must satisfy
\[
\alpha = \frac{1}{2\pi n^{1/3}} \Re L_1(x) =  \frac{1}{2\pi n^{1/3}} \Re L_2(x) 
\]
so both terms $\omega_{1,1,n}I_{1,1,n}(x)$ and $\omega_{1,2,n}(x) I_{1,2,n}(x)$ will contribute to the dominant asymptotics of $Q_n(x)$.
A similar modification must be used for the asymptotics on the compact subsets of $(x^*,0)$.  In  this case, the two terms are
 $\omega_{1,1,n}I_{1,1,n}(x)$ and its complex conjugate which combine to give an oscillatory term. Further simplification occurs because
 the argument of $L_1(x)$ is constant on $(x^*,0)$. We record the results:

\begin{theorem}
Let $x^*$ be the negative real number given in Theorem \ref{thm:phase}.
\newline
(a)
Let $x \in X \subset (x^*,0)$ be compact. Then
\begin{align*}
Q_n(x)
&=
 2^{7/6} \,
 \sqrt[12]{ 1-x} 
 \, \frac{1}{   \sqrt{ 6 \pi n^{4/3}  }} \,
 |  Li_3(x  )|^{1/6} 
{\rm exp} \left( \frac{3}{4} \sqrt[3]{2} n^{2/3} |Li_3(x)|^{1/3}  \right)
\\
&\qquad \qquad \qquad \times
\,
 \, \left( \cos
\left(
\frac{3 \sqrt{3}}{4} \sqrt[3]{2} n^{2/3} |Li_3(x)|^{1/3} + \frac{\pi}{6}
\right)+ O_X( n^{-1/3}) \right)   .
\end{align*}
(b)
Let $x \in X \subset \{ x : \Re L_1(x) = \Re L_2(x), \, x \neq x^* \}$ be compact. Then
\begin{align*}
Q_n(x)
&=
\sqrt[12]{ 1-x} \,
\sqrt{  \frac{ L_1(x) }{6 \pi  n^{4/3} }  } \exp \left( \tfrac{3}{2} n^{2/3} L_1(x)  \right)
\,  \left( 1+ O_X( n^{-1/3}) \right)  
\\
& \qquad + \,
(-1)^n
\sqrt[24]{ 1-x} \,
\sqrt[8]{  \frac{ 1-x}{1+x} } \,
\sqrt{  \frac{ L_2(x) }{6 \pi  n^{4/3} }  } \exp \left( \tfrac{3}{2} n^{2/3} L_2(x)  \right)
\,  \left( 1+ O_X( n^{-1/3}) \right)   .
\end{align*}
\end{theorem}

Note that in part (a) above, we are  abusing notation since this relation holds only if $\Re( \omega_{1,1,n}(x) I_{1,1,n}(x))$ is nonzero. On the other hand, if 
$\Re( \omega_{1,1,n}(x) I_{1,1,n}(x))=0$ holds, $Q_n(x)$ reduces to  $O( \omega_{1,1,n}(x) I_{1,1,n}(x) n^{-1/3})$.

Finally, we remark that  the asymptotics of $Q_n(x)$ over $(-1,1)$ has three separate regimes: $(-1,x^*)$, $(x^*,0)$, and $(0,1)$. On $(-1,x^*)$ and $(0,1)$, the polynomials have exponential growth
while on $(x^*,0)$ the polynomials are oscillating.  In \cite{Boyer_Parry_EJC}, this asymptotic was found to give good approximations to the zeros of $Q_n(x)$ in that interval. 
The partition polynomials built from the usual partition numbers and studied in \cite{Boyer_partition} do not have this oscillatory behavior.

\end{document}